\documentclass[12pt]{article}
\usepackage{amssymb}
\usepackage{amsmath}
\usepackage{amsthm}

\theoremstyle{plain}
\newtheorem{thm}{Theorem}

\newtheorem{claim}{Claim}

\numberwithin{equation}{section}

\title{On the Ordinary and Signed\\ G\"ollnitz-Gordon Partitions}

\author{Andrew V. Sills\\
\small{Department of Mathematical Sciences}\\
\small{Georgia Southern University}\
\small{Statesboro, Georgia, USA}\\
\small{\texttt{asills@GeorgiaSouthern.edu}}
}

\date{\small{Version of October 14, 2007}}

\begin{document}

\maketitle

\centerline{\it Dedicated to George Andrews on the occasion of his 70th birthday}
\section{Introduction}
A partition of an integer $n$ is a representation of $n$ as
an unordered sum of positive integers.  
In a recent paper~\cite{A07}, Andrews introduced the notion of a
``signed partition,"  that is, a representation of a positive integer as an
unordered sum of integers, some possibly negative.

  Consider the following $q$-series identity:
\begin{thm}[Ramanujan and Slater] For $|q|<1$,
\begin{equation} \label{RamSlat}
 \sum_{j=0}^\infty \frac{ q^{j^2} (1+q)(1+q^3) \cdots (1+q^{2j-1}) }
{ (1-q^2)(1-q^4)\cdots (1-q^{2j}) }
=\underset{m\equiv 1,4,7\hskip -3mm\pmod{8}}{ \prod_{m\geqq 1}} \frac{1}{1-q^m}.
\end{equation}
\end{thm}

 An identity equivalent to~\eqref{RamSlat} was recorded by Ramanujan
in his lost notebook~\cite[Entry 1.7.11]{AB07}.  The first proof of~\eqref{RamSlat}
was given by Slater~\cite[Eq. (36)]{S52}.

 Identity~\eqref{RamSlat} became well known after B. Gordon~\cite{G65} showed
that it is equivalent to the following partition identity, which had been discovered
independently by H. G\"ollnitz ~\cite{G60}:
\begin{thm}[G\"ollnitz and Gordon] \label{GGthm}
Let $A(n)$ denote the number of partitions of $n$ into parts which are distinct, nonconsecutive
integers where no consecutive even integers appear.  Let $B(n)$ denote the number
of partitions of $n$ into parts congruent to $1$, $4$, or $7$ modulo $8$.  Then 
$A(n)=B(n)$ for all integers $n$.
\end{thm}

Andrews~\cite[p. 569, Theorem 8]{A07} provided the following alternate combinatorial interpretation
of~\eqref{RamSlat}.
\begin{thm}[Andrews] \label{Andrews}
Let $C(n)$ denote the number of signed partitions of $n$ where
 the negative parts are distinct, odd, and smaller in magnitude
than twice the number of positive parts, and the positive parts are even 
and have magnitude at least twice the number of positive parts.
Let $B(n)$ be as in Theorem~\ref{GGthm}.  Then $C(n)=B(n)$ for all $n$.
\end{thm}

\begin{proof}
The result follows immediately after rewriting the left hand side of~\eqref{RamSlat} as
\[  \sum_{j=0}^\infty \frac{ q^{2j^2} (1+q^{-1})(1+q^{-3}) \cdots (1+q^{-(2j-1)}) }
{ (1-q^2)(1-q^4)\cdots (1-q^{2j}) }.\]
See~\cite[p. 569]{A07} for more details.
\end{proof}

  The purpose of this paper is to provide a bijection between the set of ordinary 
G\"ollnitz-Gordon partitions (those enumerated by $A(n)$ in Theorem~\ref{GGthm})
and  Andrews'  ``signed G\"ollnitz-Gordon partitions" enumerated by $C(n)$ in Theorem~\ref{Andrews}.

\section{Definitions and Notations}
A \emph{partition} $\lambda$ of an integer $n$ with $j$ parts is a $j$-tuple
$(\lambda_1, \lambda_2, \dots, \lambda_j)$ where each $\lambda_i\in\mathbb Z$, 
  \[ \lambda_1 \geqq \lambda_2 \geqq \dots \geqq \lambda_j \geqq 1\] and
 \[  \sum_{k=1}^j \lambda_k = n. \]  Each $\lambda_i$ is called a \emph{part} of $\lambda$.
 The \emph{weight} of $\lambda$ is  $n=\sum_{k=1}^j \lambda_k$ and is denoted $|\lambda|$.
 The number of parts in $\lambda$ is also called the \emph{length} of $\lambda$ and is
 denoted $\ell(\lambda)$.
 
 Sometimes it is more convenient to denote a partition by
 \[ \langle 1^{f_1} 2^{f_2} 3^{f_3} \cdots \rangle \]
 meaning that the partition is comprised of $f_1$ ones, $f_2$ twos, $f_3$ threes, etc.
 
 When generalizing the notion of partitions to Andrews' ``signed partitions," i.e.
 partitions where some of the parts are allowed to be negative, it will be
 convenient to segregate the positive parts from the negative parts.  Thus
 we define
a \emph{signed partition} $\sigma$ of an integer $n$ as a pair of (ordinary) partitions
 $(\pi, \nu)$ where $n=|\pi| - |\nu|$.  The parts of $\pi$ are the \emph{positive parts}
 of $\sigma$ and the parts of $\nu$ are the \emph{negative parts} of $\sigma$.
 We may also refer to $\pi$ (resp. $\nu$) as the \emph{positive (resp. negative) subpartition} 
 of $\sigma$.
 
  Let us denote the \emph{parity function} by
 \[ P(k):= \left\{ \begin{array}{ll}
        0 &\mbox{if $k$ is even}\\
        1 &\mbox{if $k$ is odd.}
        \end{array} \right.
  \]

  Let $\mathcal{G}_{n,j}$ denote the set of partitions 
     \[ \gamma = (\gamma_1, \gamma_2, \dots, \gamma_j ) \]
  of weight $n$ and length $j$, where for $1\leqq i\leqq j-1$,
  \begin{align}
     \gamma_i - \gamma_{i+1} & \geqq 2 \\
     \gamma_i - \gamma_{i+1} &> 2 \mbox{ \hskip 3mm    if $\gamma_i$ is even.}
  \end{align}
  Thus $\mathcal{G}_{n,j}$ is the set of those partitions enumerated by $A(n)$ in 
  Theorem~\ref{GGthm} which have length $j$. 
  
  Let $\mathcal{S}_{n,j}$ denote the set of signed partitions $\sigma = (\pi, \nu)$
  of $n$
   such that
    \begin{gather}
       \ell(\pi) = j\\
       \ell(\nu) \leqq j\\
       \pi_i\mbox{  is even for all $i=1,2,\dots, j$}\\
       \pi_i \geqq 2j \mbox{  for all $i=1,2,\dots j$}\\
       \nu_i \mbox{ is odd for all $i=1,2,\dots, \ell(\nu)$}\\
        \nu_i \leqq 2j-1 \mbox{ for all $i=1,2,\dots, \ell(\nu)$}\\
        \nu_i - \nu_{i+1} \geqq 2  \mbox{  for all $i=1,2,\dots, \ell(\nu)-1$},
    \end{gather}
   i.e. the positive subpartition is a partition into $j$ even parts, all at least $2j$, and 
   the negative subpartition is
   a partition into distinct odd parts, all less than $2j$.  Thus $\mathcal{S}_{n,j}$ is
   the set of those signed partitions enumerated by $C(n)$ in Theorem~\ref{Andrews} which
   have exactly $j$ parts.
 
 \section{A bijection between ordinary and signed G\"ollnitz-Gordon partitions}
 
 \begin{thm} The map
 \[ g: \mathcal{G}_{n,j} \to \mathcal{S}_{n,j} \]
 given by
 
 \[ (\gamma_1, \gamma_2, \dots, \gamma_j ) \overset{g}{\mapsto } 
 \Big(  (\pi_1, \pi_2, \dots, \pi_j), \langle 1^{f_1} 3^{f_3} \cdots (2j-1)^{f_{2j-1}}  \rangle  \Big) 
 \]
 where \[ \pi_k = \gamma_k + 4k - 2j -2 + P(\gamma_k) + 2\sum_{i=k+1}^j P(\gamma_i) \]
 and \[ f_{2k-1} = P(\gamma_k) \]
 is a bijection.
\end{thm}

\begin{proof} Suppose that $\gamma\in\mathcal{G}_{n,j}$ and that the image of $\gamma$
under $g$
is the signed partition $\sigma=(\pi,\nu)$.
  \begin{claim}\label{c0} $|\sigma| = |\pi|-|\nu|  = n$. \end{claim}
  \begin{proof}[Proof of Claim~\ref{c0}]
    \begin{align*}
        |\pi| - |\nu| &= \sum_{k=1}^j \left(\gamma_k + 4k - 2j - 2 + P(\gamma_k) + 
           2\sum_{i=k+1}^j P(\gamma_i) \right) \\
           &\qquad - \left( \sum_{h=1}^j (2h-1) P(\gamma_h) \right)\\
       &= \left( \sum_{k=1}^j \gamma_k \right) + 4\frac{ j(j+1)}{2} - 2j^2 - 2j 
          + \sum_{k=1}^j P(\gamma_k)\\
      &\qquad  + 2 \sum_{k=1}^j \sum_{i=k+1}^j P(\gamma_i) - \left( \sum_{h=1}^j (2h-1) P(\gamma_h) \right)\\
      &= n -\sum_{h=1}^j (2h-2) P(\gamma_h) + 2 \sum_{i=1}^j (h-1) P(\gamma_h) \\
      &=n
    \end{align*}
  \end{proof}
  \begin{claim} \label{c1} $\pi_1 \geqq \pi_2 \geqq \dots \geqq \pi_j$. \end{claim}
  \begin{proof}[Proof of Claim~\ref{c1}]  
       Fix $k$ with $1\leqq k < j$.
      \begin{align*}  
         \pi_k - \pi_{k+1} &=  \gamma_k + 4k - 2j -2 + P(\gamma_k) + 2\sum_{i=k+1}^j P(\gamma_i)\\
          & \hskip 3mm 
          -\left( \gamma_{k+1} + 4(k+1) - 2j -2 + P(\gamma_{k+1}) + 2\sum_{i=k+2}^j P(\gamma_i) \right)\\
          &= \gamma_k - \gamma_{k+1} + P(\gamma_k) - P(\gamma_{k+1}) - 4.
      \end{align*}
      The minimum value of $\gamma_k - \gamma_{k+1}$ varies depending on the parities of
      $\gamma_k$ and $\gamma_{k+1}$.
      \begin{itemize}
         \item If $\gamma_k \equiv \gamma_{k+1} \equiv 0 \pmod{2}$, then 
         $$(\gamma_k - \gamma_{k+1}) + P(\gamma_k) - P(\gamma_{k+1}) - 4 \geqq 4+ 0 + 0 -4=0.$$
          \item If $\gamma_k \equiv 1\pmod{2}$ and $ \gamma_{k+1} \equiv 0 \pmod{2}$, then 
           $$(\gamma_k - \gamma_{k+1}) + P(\gamma_k) - P(\gamma_{k+1}) - 4 \geqq 3+ 1 + 0 -4=0.$$
         \item If $\gamma_k \equiv 0\pmod{2}$ and $ \gamma_{k+1} \equiv 1 \pmod{2}$, then 
           $$(\gamma_k - \gamma_{k+1}) + P(\gamma_k) - P(\gamma_{k+1}) - 4 \geqq 3+ 0 + 01 -4=0.$$
           \item If $\gamma_k \equiv \gamma_{k+1} \equiv 1 \pmod{2}$, then 
         $$(\gamma_k - \gamma_{k+1}) + P(\gamma_k) - P(\gamma_{k+1}) - 4 \geqq 2+ 1+ 1 -4=0.$$
         \end{itemize}
      \end{proof}      
      \begin{claim}\label{c2} All of the $\pi_k$ are at least $2j$. \end{claim}
        \begin{proof}[Proof of Claim~\ref{c2}]  By Claim~\ref{c1}, it is sufficient to show that $\pi_j\geqq 2j$.  
         
           If $\gamma_j=1$, then 
         \begin{align*} 
         \pi_j &= \gamma_j + 4j - 2j -2 + P(\gamma_j) + 2\sum_{i=j+1}^j P(\gamma_i) \\
         &= \gamma_j +2j - 2 + 1 \\ 
         & \geqq 1+2j-2+1\\ 
         &=2j. 
         \end{align*} 
         Otherwise $\gamma_j\geqq 2$, and so
             \begin{align*} 
         \pi_j &= \gamma_j + 4j - 2j -2 + P(\gamma_j) + 2\sum_{i=j+1}^j P(\gamma_i) \\
         &= \gamma_j +2j - 2 + 1 \\ 
         & \geqq 2+2j-2+0\\ 
         &=2j. 
         \end{align*} 
        \end{proof}
        
     \begin{claim}\label{c3}  All parts of $\pi$ are even.\end{claim}
     \begin{proof}[Proof of Claim~\ref{c3}]  
       \begin{align*}  
         \pi_k - \pi_{k+1} &=  \gamma_k + 4k - 2j -2 + P(\gamma_k) + 2\sum_{i=k+1}^j P(\gamma_i)\\
          &\equiv \gamma_k + P(\gamma_k) \\
          &\equiv 0 \pmod {2}.
      \end{align*}
      \end{proof}  
      
      \begin{claim}\label{c4} All parts of $\nu$ are distinct, odd, and at most $2j-1$.\end{claim}
      \begin{proof}[Proof of Claim~\ref{c4}]  Claim~\ref{c4} is clear from the definition of $g$
      together with the observation that $P(\gamma_i)\in \{0,1 \}$ for any $i$. 
      \end{proof}
      
      \begin{claim}\label{c5} The map $g$ is invertible.
      \begin{proof}[Proof of Claim~\ref{c5}]
        Let  \[ h: \mathcal{S}_{n,j} \to \mathcal{G}_{n,j} \] be given by
         \[  \Big( (\pi_1, \pi_2, \dots, \pi_j) , \langle  1^{f_1} 3^{f_3} \cdots (2j-1)^{f_{2j-1}} \rangle \Big) 
         \overset{h}{\mapsto } 
        (\gamma_1, \gamma_2, \dots, \gamma_j)
 \]
 where \[ \gamma_k = \pi_k - 4k + 2j + 2 - f_{2k-1} - 2\sum_{i=k+1}^j f_{2i-1} \]
   for $1\leqq k \leqq j$.
   Direct computation shows that $h(g(\gamma)) = \gamma$ for all $\gamma\in\mathcal{G}_{n,j}$,
   and $g(h(\sigma)) = \sigma$ for all $\sigma\in\mathcal{S}_{n,j}$.  Thus $h$ is the inverse of $g$.
      \end{proof}
      \end{claim}
    Hence, by the above claims $g$ is a bijection.
\end{proof}

\end{document}